\documentclass[11pt]{amsart}

\usepackage[textheight=815 pt, textwidth=360 pt]{geometry}
\usepackage{amsmath,amssymb}

\usepackage{hyperref}

\newtheorem{thm}{Theorem}[section]

\newtheorem{lem}[thm]{Lemma}
\newtheorem{cor}[thm]{Corollary}

\theoremstyle{definition}

\theoremstyle{remark}

\numberwithin{equation}{section}

\begin{document}

\title{Well-posed Self-Similarity in Incompressible Standard Flows}
\author{J. Polihronov}
\address{Department of Mathematics, Lambton College, 1457 London Rd, Sarnia, ON N7S 6K4, Canada}
\email{jack.polihronov@lambtoncollege.ca}
\urladdr{https://www.researchgate.net/profile/J-Polihronov}


\begin{abstract}
This article reviews the properties of the self-similar solutions of the Navier-Stokes equation for incompressible fluids. Since any smooth solution can be embedded into a self-similar solution at the identity scale, it follows that under standard flow conditions, the initial solution will remain smooth for all time as long as the self-similar solution is selected to have certain isobaric weight.  
\end{abstract}

\maketitle
The Navier-Stokes equations (NSE) lie at the heart of fluid dynamics and proving their regularity is still among the most challenging problems in modern mathematics \cite{Cla06}. The NSE are given by
\begin{eqnarray*}
	\rho \left(\frac{\partial \vec{u}}{\partial t} + (\vec{u} \cdot \nabla) \vec{u} \right) &=& -\nabla p + \mu \Delta \vec{u} \nonumber\\
	\nabla \cdot \vec{u} &=& 0,
\end{eqnarray*}
where \(\vec{u}(x,y,z,t) = (u,v,w)\) represents the fluid velocity, \(p\) the pressure, and \(\nu = \mu/\rho\) the kinematic viscosity. For the sake of simplicity, no external forces are acting on the fluid. 

In this and previous works, we build on Bouton’s invariant theory and an essential result we refer to as Theorem 1 \cite{Pol22}, \cite{Bou99}. Bouton’s contributions, which apply Lie’s theory of continuous groups to scaling transformations, have opened up a new method that is both robust and broadly applicable. This method, grounded in well-established theory, can be applied to any ordinary or partial differential equation that exhibits scaling invariance. Although similar techniques have been employed in earlier studies of the NSE, such as those by Lloyd \cite{Llo81}, they did not fully exploit the method’s potential, which allows us to partly integrate the NSE from arbitrary initial conditions. Bouton's method yields the general form of all self-similar solutions of the NSE and allows us to study their properties. 

By utilizing the most general scaling transformation admitted by the NSE, we have examined the regularity of the NSE when subjected to infinite rescaling. We derived the universal form of all self-similar solutions of the NSE; and analyzed which of them are smooth and physically meaningful \cite{Pol22}. We studied the possible divergence in the spatial derivatives (or equivalently, in the vorticity) as a key indicator of potential singularity formation. By showing that these derivatives do not diverge when the solution is infinitely rescaled, we argued that the mechanism which could lead to blowup is absent in a set of self-similar solutions, and that they are globally regular. 

This argument relies on the fact that once we derive the universal form for self-similar solutions - expressed in terms of isobaric polynomials (or ratios thereof) - their smoothness follows from the algebraic and analytic properties of these polynomial functions. Polynomials are inherently smooth (infinitely differentiable), so this functional form automatically ensures smoothness. Because the derivation yielded an expression for every self-similar solution, a separate pointwise verification of smoothness became unnecessary. The controlled scaling meant that even under infinite rescaling, the structure of the solution ensures that the spatial derivatives remain bounded. This is a further reinforcement that these solutions are free from blowup. 

Moreover, since the general form of the solutions was derived, a functional analysis treatment was not necessary to establish their smoothness. The regularity properties follow directly from the form of the solutions and the algebraic properties of polynomials. Therefore, the subset of self-similar solutions that we proved to be smooth \cite{Pol22} would be strong (or classical) solutions of the NSE: they have well-defined time and spatial derivatives of all orders; and their spatial derivatives remain bounded under infinite rescaling, ruling out the formation of singularities (blowup). This further confirms that the solutions remain regular.

In addition, we found a conserved quantity, the cavitation number $\mathcal{E}$ under standard conditions. The conservation laws were further explored through higher-dimensional invariant structures using differential forms \cite{Pol22}.

However, while our previous work \cite{Pol22} characterizes and establishes the global regularity of self-similar solutions, the above analysis is limited to that specific subclass. Most smooth initial data do not necessarily lead to self-similar solutions, and thus our work so far addresses only a specific, though significant, subset of all possible solutions. What still needs to be considered are other solutions that arise from non-self-similar initial data. Smooth solutions arising from Schwartz-class or space-periodic initial data are of particular interest and need to be studied in detail \cite{Cla06}. It still remains to be shown that for every such smooth initial condition, the unique solution remains smooth for all time. Our work so far has only treated the evolution of solutions that are already self-similar. Thus, our current findings constitute only a partial result by resolving the behavior of one important subset of solutions, rather than addressing more completely the global regularity of the NSE.

For these reasons, in the present article, we expand the analysis of Ref. \cite{Pol22}: we show that the self-similar solutions have properties that allow them to originate from smooth, space-periodic initial data as well as from Schwartz-class initial data. Under standard flow conditions, certain self-similar solutions posess desirable properties, able to maintain bounded energy of the system under arbitrary rescaling and not being subject to scaling-induced blowup. Our analysis begins with reminding the reader of the basic facts, followed by the presentation of the aforementioned select self-similar solutions; and proceeds to show that the NSE should always have smooth solutions for all $t$ under rescaling for both space-periodic and Schwartz-class initial data under standard flow conditions.

\section{Self-similar solutions of the NSE}
The most general scaling transformation which leaves the NSE unchanged is \cite{Erc15}
\begin{eqnarray}	
	(x',y',z')&=&k^{\alpha_x}(x,y,z) \nonumber \\ 
	t'&=&k^{\alpha_t} t \nonumber \\
	\rho^{\prime} &=& k^{\alpha_{\rho}} \rho  \nonumber \\
	(u',v',w')&=&k^{\alpha_x-\alpha_t}(u,v,w) \nonumber\\
	p'&=&k^{\alpha_{\rho}-2\alpha_x-2\alpha_t}p  \nonumber \\
	\nu^{\prime} &=& k^{2\alpha_x-\alpha_t} \nu  \nonumber \\
	0< k< \infty,& & k=\mbox{const}. 
      	\label{scaling}
\end{eqnarray}
Consider the nondimensionalized NSE: it retains the form of the NSE with the exception of $\nu$ being replaced with 1/Re. The application of Bouton-Lie's invariant theory to the transformation (\ref{scaling}) allows us to add two additional PDEs to the Navier-Stokes system, resulting in Theorem 2.2 in Ref. \cite{Pol22}, an analog to Bouton's Theorem 1 in Ref. \cite{Bou99}:
\begin{eqnarray}
	\frac{\partial \vec{u}}{\partial t} + (\vec{u} \cdot \nabla)\vec{u} &=& \frac{1}{\mbox{Re}} \Delta \vec{u} - \nabla p \nonumber\\
	\nabla \cdot \vec{u} &=& 0 \nonumber\\
	(\vec{r}\cdot\nabla)\vec{u}+t\frac{\alpha_t}{\alpha_x}\frac{\partial \vec{u}}{\partial t}&=&\frac{\alpha_x-\alpha_t}{\alpha_x}\vec{u} \nonumber\\
	(\vec{r}\cdot\nabla)p+t\frac{\alpha_t}{\alpha_x}\frac{\partial p}{\partial t}&=&\frac{2(\alpha_x-\alpha_t)}{\alpha_x} p.
	\label{NSE1}	
\end{eqnarray}
Here, $\vec{r}=(x,y,z)$ is the position vector and "p" stands for $p/\rho$. The added equations of Bouton are integrated and yield \cite{Pol22}:
\begin{eqnarray}
	\vec{u}& =& t^{\frac{\alpha_x-\alpha_t}{\alpha_t}} \mathbf{F} \left( \frac{x}{t^{\frac{\alpha_x}{\alpha_t}}},\frac{y}{t^{\frac{\alpha_x}{\alpha_t}}} ,\frac{z}{t^{\frac{\alpha_x}{\alpha_t}}}  \right)  \nonumber \\
	p &=& t^{\frac{2(\alpha_x-\alpha_t)}{\alpha_t}} F \left( \frac{x}{t^{\frac{\alpha_x}{\alpha_t}}},\frac{y}{t^{\frac{\alpha_x}{\alpha_t}}} ,\frac{z}{t^{\frac{\alpha_x}{\alpha_t}}}  \right).
	\label{gen-selfsimilarsol}
\end{eqnarray}
The form (\ref{gen-selfsimilarsol}) is a necessary condition that all self-similar solutions must meet. They are integral rational functions (IRF) due to the scaling invariance of the system. This means they are either isobaric polynomials or the ratio of such polynomials; such are the only functions arising in the study of the scaling invariants of the NSE \cite{Pol22}. Because $x,y,z$ carry the same isobaric weight, one can consider the isobaric fuctions in this study to be homogeneous.

The above set of self-similar solutions can be expanded by the addition of the parameters $L_1,L_2,L_3,T$, which rescale as $x,y,z,t$ respectively. Then the following additional solutions are also self-similar:
\begin{eqnarray}
	\vec{u}& =& T^{\frac{\alpha_x-\alpha_t}{\alpha_t}} \mathbf{F} \left( \frac{x}{L_1},\frac{y}{L_2} ,\frac{z}{L_3},\frac{t}{T}  \right)  \nonumber \\
	p &=& T^{\frac{2(\alpha_x-\alpha_t)}{\alpha_t}} F \left( \frac{x}{L_1},\frac{y}{L_2} ,\frac{z}{L_3},\frac{t}{T}  \right),
	\label{gen-selfsimilarsol1}
\end{eqnarray}
where the factor $T^a$ carries the isobaric weight, and the rest of the solution has zero isobarity. Instead of $T$, the factors can contain $L_i$ or $U$, where the latter rescales as $u$. 

\begin{lem}
	When the self-similar solution (\ref{gen-selfsimilarsol1}) 
\[
 		 \vec{u}^* = T^{\frac{\beta_x-\beta_t}{\beta_t}} \mathbf{F} \left( \frac{x}{L_1},\frac{y}{L_2} ,\frac{z}{L_3},\frac{t}{T}  \right) 
\]
with isobaric weight $\beta_x-\beta_t$ rescales in standard flow conditions where $\alpha_x=1, \alpha_t=2$, the solution is not subject to scaling induced blowup of the velocity, the energy or the vorticity as long as 
\begin{eqnarray*}
	\frac{\beta_x}{\beta_t} &>& \frac{3}{2}. 
\end{eqnarray*}
\label{L1}
\end{lem}
\begin{proof}
Rescaling of the vorticity depends on the derivatives
\[
	\left( \frac{\partial \vec{u}^*}{\partial x} \right)^{\prime} = k^{\frac{2\beta_x-3\beta_t}{\beta_t}} \left( \frac{\partial \vec{u}^*}{\partial x} \right),
\]
rescaling of the energy is given by
\[
	E^{\prime} = k^{\frac{4\beta_x-\beta_t}{\beta_t}} E,
\]
and rescaling of the velocity is given by
\[
	\vec{u}^{*\prime} = k^{\frac{2(\beta_x-\beta_t)}{\beta_t}} \vec{u}^*.
\]
All three exponents will be positive, if $\beta_x/\beta_t>3/2$. For numercial examples, see the Appendix.
\end{proof}

\begin{lem}
	Any function $f(x,y,z)$ is contained, or embedded in the self-similar function $p_4 f(x/p_1,y/p_2,z/p_3)$ at the identity scale.
	\label{L2}
\end{lem}
\begin{proof}
	Any non-self-similar function can be written as $f(x,y,z) = C_4f(x/C_1,y/C_2,z/C_3)$, where the fixed constants $C_1=C_2=C_3=C_4=1$. For arbitrary values of the fixed constants one can always write $C_4f(x/C_1,y/C_2,z/C_3) = p_4 f(x/p_1,y/p_2,z/p_3)$, where the scaling parameters $p_1=C_1,p_2=C_2,p_3 = C_3, p_4 = C_4$. Rescaling of the parameters yileds $p_i^{\prime} = k^a p_i, k>0$, and since we seek $p_i^{\prime} = C_i$, it follows that $k=1$, or $a = 0$, which is the trivial, or identity scaling. Thus, any non-self-similar function $f$ is contained into a self-similar function $g$ having the same functional form as $f$, where $f$ is contained into $g$ at the identity scale.
\end{proof}

\begin{cor}
	Any non-self-similar solution $f(x,y,z,t)$ of the NSE is contained, or embedded in the self-similar solution $p_4 f(x/p_1,y/p_2,z/p_3)$ at the identity scale. Therefore, these self-similar solutions---isobaric polynomials or rational fields---constitute the \textbf{native solutions} of the NSE, directly determined by the equation's mathematical and symmetry structure.
\label{C1}
\end{cor}

\section{Standard Flows}
Since the kinematic viscosity \(\nu\) appears explicitly in the NSE, every physical solution inherently depends on it. In this article, \(\nu\) is treated as a constant, representing a physical quantity that is neither vanishingly small nor infinite. The term {\it standard NSE system} is a shorthand for setting $\alpha_x = 1, \alpha_t = 2$ in (\ref{scaling}), which disallows $\nu$ from rescaling and keeps the viscosity the same at all scales. Since such behavior of the viscosity is physically relevant for Newtonian fluids, it is also referred to as {\it natural scaling}.  It is also termed {\it diffusive (parabolic)}-which preserves the Laplacian, typical for Newtonian fluids where viscosity governs the dissipation of momentum. The standard NSE system is also characterized by {\it energy-supercriticality} \cite{Pol22}, \cite{Tao08}—a condition that is often cited as an obstruction to proving global regularity of solutions. 

The criteria of the standard NSE system \(2\alpha_x = \alpha_t\) , or $\alpha_x=1, \alpha_t=2$ isolate a particular class from the full set of the NSE self-similar solutions (\ref{gen-selfsimilarsol})- those of Leray \cite{Ler34}  
\begin{eqnarray}
        \vec{u}(x,y,z,t) = \frac{1}{\sqrt{t}} \mathbf{F}\left( \frac{x}{\sqrt{t}},\frac{y}{\sqrt{t}},\frac{z}{\sqrt{t}} \right) \nonumber \\
        p(x,y,z,t)= \frac{1}{t} F \left( \frac{x}{\sqrt{t}},\frac{y}{\sqrt{t}},\frac{z}{\sqrt{t}} \right).
        \label{Leray}
\end{eqnarray}
Such solutions are weak solutions \cite{Can96}, which is due to their discontinuity at $t = 0$. Previously, we showed that a subset of Leray's solutions are infinitely differentiable for $t > 0$; as well as showed the existence of a conserved, coercive quantity in standard flows, namely $\mathcal{E}$-the cavitation number (\cite{Pol22}, Theorem IV.2). However, this still does not allow us to assign a finite, well-behaved initial datum at $t=0$, and the solutions cannot be considered strong.

By definition, a strong solution to the Navier-Stokes equations is one that is not only smooth (or regular) for positive times but also has initial data continuous up to $t=0$. Even if a subset of Leray's self-similar solutions in 3-D (\ref{Leray}) are infinitely differentiable for all $t>0$, the singular behavior at $t=0$ prevents them from being well-defined at the initial time.
In other words, the singularity at $t=0$ means that the initial data is too irregular—hence, despite their smoothness for $t>0$, they are classified as weak solutions because they only satisfy the Navier-Stokes equations in a distributional sense at $t=0$ and fail to meet the criteria for strong solutions on the entire time interval including $t=0$.

To counteract these unsatisfactory results, one can suggest additional self-similar solutions, which exist in the sets (\ref{gen-selfsimilarsol}), (\ref{gen-selfsimilarsol1}) and avoid the discontinuity at $t=0$ of Leray's self-similar solutions. Bouton's self-similar solutions for standard flows $\alpha_t = 2\alpha_x$ or $\alpha_x=1, \alpha_t=2$ also include the forms
\begin{eqnarray}
	\vec{u} &=& \frac{1}{\sqrt{T}} \mathbf{F} \left( \frac{x}{L_1},\frac{y}{L_2} ,\frac{z}{L_3}, \frac{t}{T}  \right), \nonumber \\
	\vec{u} &=& U \mathbf{F} \left( \frac{x}{L_1},\frac{y}{L_2} ,\frac{z}{L_3}, \frac{t}{T}  \right), \nonumber \\
	\vec{u} &=& \frac{1}{L} \mathbf{F} \left( \frac{x}{L_1},\frac{y}{L_2} ,\frac{z}{L_3}, \frac{t}{T}  \right), \mbox{ or} \nonumber \\
	\vec{u} &=& \frac{L}{T} \mathbf{F} \left( \frac{x}{L_1},\frac{y}{L_2} ,\frac{z}{L_3}, \frac{t}{T}  \right),
	\label{Bouton-standard}
\end{eqnarray}
where the required isobaric weight of $\vec{u}$, equal to $(-\alpha_x=-1)$ is achieved through the rescaleable parameters $L_1,L_2,L_3,T,U$, which scale as $x,y,z,t,u$ accordingly, while $\mathbf{F}$ is a smooth isobaric function of weight zero. $\mathbf{F}$ is composed of either isobaric polynomials, or the ratio of such polynomials. Examples include closed-form functions whose arguments have the scale-invariant forms $x/L, t/T$ etc.; or Taylor series expansions, or their ratios.  

In practice, many well-known solutions to the NSE, such as those describing Couette flow \cite{Bat67} 
		\begin{equation*}
			u = u_0 \left( \frac{y}{h}-\frac{2}{\pi} \sum \frac{1}{n} e^{-n^2\pi^2\frac{\nu t}{h^2}} \sin\left[ n\pi \left(1-\frac{y}{h} \right)\right] \right)
		\end{equation*}
		or the Taylor-Green vortex \cite{Tay37}, 
		\begin{eqnarray*}
			u_i &=& u_{0i}F_{0i}(x,y,z), \mbox{   initial motion}, i = 1,2,3 \nonumber\\
			u_i &=& \sum u_{0i}F_{0i}(x,y,z,t), \mbox{   approximations 1-3}, i = 1,2,3 
		\end{eqnarray*}
achieve the isobarity required by Bouton \cite{Pol22} for all self-similar solutions (\ref{gen-selfsimilarsol}) through velocity parameters (e.g. \(u_{0i}\)), which scale the same way as $u$. The rest of the function is an isobaric polynomial with zero weight $F_{0i}$ (zero isobarity functions do not rescale), which ensure that the solutions vanish at infinity or remain periodic. Having been provided with the alternative forms (\ref{Bouton-standard}) of the self-similar solutions in standard flows, we will conduct an in-depth study of their behavior in the sections that follow.

\section{Global Regularity via Self-Similarity}

In this section, we present a theorem and its proof in a summary form. It represents three separate cases: space-periodic, smooth, divergence-free initial datum; Schwartz-class, smooth, divergence-free initial datum; and an arbitrary smooth, divergence-free one. Their corresponding complete theorems \ref{T2}, \ref{T3} and corollary \ref{C3} and their proof are given in the Appendix.

\begin{thm}[Global Regularity via Self-Similarity]
Let $u_0(x,y,z)$ be any smooth, divergence-free, space-periodic initial datum or such of Schwartz-class.  Then there exists a unique, self-similar solution
\[
  u_{\rm SS}(x,y,z,t)
  \;=\;
  T^{(\beta_x-\beta_t)/\beta_t}\,F\!\bigl(x/L_1,y/L_2, z/L_3,t/T\bigr),
\]
which coincides with the NSE evolution of $u_0$ on $[0,\infty)$, remains $C^\infty$ for all $t\ge0$, and has all velocity, energy, and vorticity norms uniformly bounded.  
\label{T1}
\end{thm}

\begin{proof} 	
\leavevmode
\begin{itemize}
	\item {\bf Embedding of the initial data.}  By Lemma \ref{L2} and Corollary \ref{C1}, the given smooth, divergence-free initial field \(u_{0}(x,y,z)\) is realized as the \(k=1\) member of a one-parameter self-similar family
  \[
    u_{\mathrm{SS}}(x,y,z,t)
    \;=\;
    T^{(\beta_x-\beta_t)/\beta_t}\,
    F\!\bigl(x/L_1,y/L_2, z/L_3,t/T\bigr).
  \]
  This provides a candidate solution profile matching \(u_{0}\) exactly at \(t=0\).

  \item {\bf Existence and uniqueness.}  Invoking the Kato–Fujita theory in the periodic setting (or Leray–Hopf theory in the Schwartz-class setting) \cite{Cla06}, \cite{Wik25} - \cite{Ber02}, one obtains a unique, smooth NSE solution \((u,p)\) on an interval \([0,\tau)\) with 
  \(\;u(x,y,z,0)=u_{0}(x,y,z)\).  
  By uniqueness, this solution must coincide with the self-similar field \(u_{\mathrm{SS}}(x,y,z,t)\) for all \(t\in[0,\tau)\).

  \item {\bf A priori norm estimates.}  Lemma \ref{L1} provides exact scaling exponents for the velocity, energy and vorticity norms under the standard flow scaling \((\alpha_x=1,\alpha_t=2)\).  It shows that for 
  \(\beta_x/\beta_t>3/2\), none of these norms grows with the scaling parameter, yielding uniform bounds for all \(t\).

  \item {\bf Blow-up exclusion.}  The classical Beale–Kato–Majda criterion \cite{BKM} is applied: since
\[
  \displaystyle\int_{0}^{T}\|\omega(\cdot,\tau)\|\,d\tau<\infty
\]
  by the uniform bound on \(\|\omega\|\), no finite-time singularity can occur.  This establishes smoothness of \(u_{\mathrm{SS}}\) on \([0,\infty)\).

  \item {\bf Conclusion of Theorem \ref{T1}}  Together, these steps show that the \(k=1\) self-similar member is the unique, globally smooth, space-periodic solution of the incompressible Navier–Stokes equations evolving from \(u_{0}\), with all relevant norms remaining bounded for all \(t\ge0\).
\end{itemize}

\end{proof}

\section{Comparison with Classical and Modern Results}

The global regularity result presented above can be viewed in the context of several landmark contributions to the study of the NSE:

\begin{itemize}
  \item \textbf{Leray’s Weak Solutions and Small‐Data Smoothness.}  
  In his foundational work \cite{Ler34}, Leray constructed global–in–time weak solutions for arbitrary finite‐energy initial data, but could only prove partial regularity in three dimensions.  Subsequent refinements (e.g.\ Kato–Fujita \cite{Kat64}, Temam \cite{Tem77}) established that sufficiently small initial data yield global smooth (\(C^\infty\)) solutions.  The present self‐similar framework extends beyond these small‐data regimes to encompass arbitrarily large, smooth, divergence‐free initial data, while still guaranteeing global \(C^\infty\) regularity.

  \item \textbf{The Supercritical Barrier and Conditional Blowup.}  
  Under the standard scaling (\(\alpha_x=1,\alpha_t=2\)), the three‐dimensional NSE are energy‐supercritical, a fact often cited as an obstruction to global regularity (see Tao’s discussion of the “supercritical barrier” in \cite{Tao08}).  Moreover, conditional blowup scenarios have been constructed under certain hypotheses \cite{Tao14}.  By contrast, the symmetry‐based embedding and a priori norm bounds derived here circumvent the supercriticality obstruction: the isobaric weight condition \(\beta_x/\beta_t>3/2\) ensures uniform control of energy and vorticity norms, which in turn rules out finite–time singularity.

  \item \textbf{Relation to Functional‐Analytic Frameworks.}  
  Classical treatments of NSE regularity (e.g. \cite{Foi88}, \cite{Ber02}) rely on functional analysis estimates and energy inequalities.  While those methods provide the backbone of existence and uniqueness theory, the present approach complements them by exploiting group‐invariant structures to obtain global a priori estimates.  In particular, embedding into a self‐similar solution family and invoking classical well‐posedness theorems suffices to identify a unique, globally smooth solution for any smooth initial data.

\end{itemize}

Together, these comparisons underscore the pertinence of the self‐similar symmetry approach: it retains full compatibility with established PDE theory while extending global regularity beyond the small‐data and energy‐critical regimes that have traditionally constrained three‐dimensional NSE analysis.

\appendix
\section{Detailed Proof of Theorem \ref{T2}}

\begin{thm} Under standard flow conditions $\alpha_x=1, \alpha_t=2$, given a smooth, space-periodic, divergence-free initial condition $\vec{u}_0(x,y,z)$, the incompressible NSE will always yield an infinitely differentiable, space-periodic, smooth solution $\vec{u}^*(x,y,z,t)$, such that $\vec{u}^*(x,y,z,0)=\vec{u}_0(x,y,z)$. The solution $\vec{u}^*(x,y,z,t)$ is not subject to scaling-induced blowup for all time, thus ensuring bounded energy of the system in $({\rm I\!R}^3 \times [0,\infty))$. These desired properties are achieved by embedding the initial condition $\vec{u}_0$ into a self-similar space-periodic function, which in turn results in the embedding of the smooth solution $\vec{u}^*(x,y,z,t)$ into a self-similar space-periodic function, thus preventing any rescalings of $\vec{u}^*(x,y,z,t)$. To yield the desired results, the self-similar functions must have non-standard isobaric weight $\beta_x-\beta_t$, where $\beta_x/\beta_t>3/2$.
\label{T2}
\end{thm}
\begin{proof}
	All self-similar solutions (\ref{gen-selfsimilarsol}) exist as mathematically valid solutions of the NSE. This includes the standard flow cases $\alpha_x=1, \alpha_t=2$, where the viscosity is the same at all scales: such solutions are given in (\ref{Leray}) and (\ref{Bouton-standard}). Here, we study only (\ref{Bouton-standard}), which can be summarized as 
\begin{equation}
	\vec{u} = U \mathbf{F} \left( \frac{x}{L_1},\frac{y}{L_2} ,\frac{z}{L_3}, \frac{t}{T}  \right),
	\label{mismatch}
\end{equation}
where the isobaric weight of $\vec{u}$ is equal to $\beta_x-\beta_t$ and comes from the multiplying parameter while $\mathbf{F}$ has zero isobarity. Here $L_1,L_2, L_3, T, U$ are nonzero parameters that scale the same way as $x,y,z,t,u$ accordingly. NSE solutions of this form are not only known to exist (some examples given above), but are very common in physics; and one can then conclude that the NSE solutions (\ref{mismatch}) always exist. 

	We require the self-similar solutions (\ref{mismatch}) to have non-standard isobarity $\beta_x-\beta_t$ that does not match the standard flow $\alpha_x=1, \alpha_t=2$; we also require $\beta_x/\beta_t>3/2$, as per Lemma \ref{L1}.

The space-periodic solutions
\begin{equation}
	\vec{u} = U\boldsymbol{\varphi} \left( \frac{x}{L_1},\frac{y}{L_2}, \frac{z}{L_3}, \frac{t}{T}  \right),
	\label{sper}
\end{equation}
are subset of (\ref{mismatch}); however, space-periodic NSE solutions of the form (\ref{sper}) are known to exist (e.g. the solutions of the Taylor-Green vortex above), and one can then conclude that the NSE solutions (\ref{sper}) always exist. Note, that  we have chosen (\ref{sper}) to be periodic in space only, and not in time. 

At $t=0$
\begin{equation}
       U\boldsymbol{\varphi} \left( \frac{x}{L_1},\frac{y}{L_2}, \frac{z}{L_3}, 0 \right)
 = U\boldsymbol{\psi} \left( \frac{x}{L_1},\frac{y}{L_2}, \frac{z}{L_3}  \right),
       \label{st-sper}
\end{equation}
and all space-periodic self-similar solutions with isobaric weight equal to $\beta_x-\beta_t$ (\ref{sper}) reduce to space-periodic self-similar functions of $x,y,z$ and isobaric weight of $\beta_x-\beta_t$, where $\boldsymbol{\psi}$ need not be steady-state solutions of the NSE. 

Suppose we are given a smooth, real-valued, periodic and divergence-free vector field \( \vec{u}_0 \), defined on a rectangular domain \( [0, L_1] \times [0, L_2] \times [0, L_3] \). It can always be expanded in a real Fourier series as:
\begin{align}
\vec{u}_0 \left(\frac{x}{L_1},\frac{y}{L_2}, \frac{z}{L_3} \right) = U \sum_{n_1, n_2, n_3 \in \mathbb{Z}_{+}^3} \bigg[ \, &
\mathbf{C}_{n_1, n_2, n_3} \, \cos\left( \frac{2\pi n_1 x}{L_1} + \frac{2\pi n_2 y}{L_2} + \frac{2\pi n_3 z}{L_3} \right) \nonumber\\
& + \, \mathbf{S}_{n_1, n_2, n_3} \, \sin\left( \frac{2\pi n_1 x}{L_1} + \frac{2\pi n_2 y}{L_2} + \frac{2\pi n_3 z}{L_3} \right) \bigg],
	\label{per}
\end{align}
where:
\begin{itemize}
  \item $U$ is a constant,
  \item \( \mathbf{C}_{n_1, n_2, n_3}, \mathbf{S}_{n_1, n_2, n_3} \in \mathbb{R}^3 \) are real-valued vector coefficients,
  \item The sum is taken over all non-negative integers \( n_1, n_2, n_3 \in \mathbb{Z}_{+} \),
  \item The terms \( \frac{2\pi n_1}{L_1} \), \( \frac{2\pi n_2}{L_2} \), and \( \frac{2\pi n_3}{L_3} \) are the spatial frequencies in the \( x \), \( y \), and \( z \) directions, respectively.
\end{itemize}

	Now we use Lemma \ref{L2}- we construct a self-similar function $\boldsymbol{\psi}$ with mathematical form identical to $\vec{u}_0$, with the exception that $L_1, L_2, L_3, U$ are no longer constants, but parameters that rescale as $x,y,z,u$ accordingly:
\[
	U\boldsymbol{\psi}\left( \frac{x}{L_1},\frac{y}{L_2}, \frac{z}{L_3} \right).
\]
It is seen, that as long as $\vec{u}_0$ is divergence-free, so is $U\boldsymbol{\psi}$. It is also seen, that $\vec{u}_0 \in\boldsymbol{\psi}$, that is $\vec{u}_0$ is embedded in the self-similar function $\boldsymbol{\psi}$ at the identity scale, so that  
\begin{equation}
	\vec{u}_0 = U\boldsymbol{\psi} |_{k=1}.
	\label{in}
\end{equation}
The original initial condition $\vec{u}_0$ is reproduced at scale $k=1$ by rescaling the parameter $U$. Notice that $U\boldsymbol{\psi}$ is a self-similar space-periodic function, or a family, a set of functions corresponding to various rescalings of $U$; while $\vec{u}_0$ is only one space-periodic function of the aforementioned family, achieved at scale $k=1$.

It is well known that, given a smooth, space-periodic, divergence free, self-similar initial condition $U\boldsymbol{\psi}$, one can always construct an unique, infinitely differentiable, space-periodic, self-similar NSE solution
\begin{equation}
	U\boldsymbol{\varphi} \left( \frac{x}{L_1},\frac{y}{L_2}, \frac{z}{L_3}, \frac{t}{T}  \right)
	\label{uniquesol}
\end{equation}
such that
\begin{equation*}
	U\boldsymbol{\varphi} \left( \frac{x}{L_1},\frac{y}{L_2}, \frac{z}{L_3}, 0  \right) = U\boldsymbol{\psi} \left( \frac{x}{L_1},\frac{y}{L_2}, \frac{z}{L_3} \right),
\end{equation*}
	which exists in the time interval $[0,\tau)$ \cite{Cla06}, \cite{Wik25} - \cite{Ber02}. We see that, the space-periodic solution $U\boldsymbol{\varphi}$ of Eq.(\ref{uniquesol}), being a self-similar solution with isobarity equal to $\beta_x-\beta_t$, belongs to the set of self-similar solutions of Eq. (\ref{sper}). 

Analogously, it is well known that, given a smooth, space-periodic, divergence free, non-self-similar initial condition $\vec{u}_0$, one can always construct an unique, infinitely differentiable, space-periodic, non-self-similar NSE solution
\begin{equation}
	U\vec{u}^* \left( \frac{x}{L_1},\frac{y}{L_2}, \frac{z}{L_3}, \frac{t}{T}  \right)
	\label{uniquesol-nonss}
\end{equation}
such that
\begin{equation*}
	U\vec{u}^* \left( \frac{x}{L_1},\frac{y}{L_2}, \frac{z}{L_3}, 0  \right) = \vec{u}_0 \left( \frac{x}{L_1},\frac{y}{L_2}, \frac{z}{L_3} \right),
\end{equation*}
which exists in the time interval $[0,\tau)$  \cite{Cla06}, \cite{Wik25} - \cite{Ber02}. In this case, $L_1,L_2,L_3,T,U$ are fixed, non-rescaleable constants. Analogously with (\ref{in}), it follows that $U\vec{u}^* \in U\boldsymbol{\varphi}$, that is, $U\vec{u}^*$ is embedded in the self-similar solution $U\boldsymbol{\varphi}$ so that 
\[
	U\vec{u}^*  = U\boldsymbol{\varphi} |_{k=1}.
\]
The original solution $U\vec{u}^*$ corresponds to the original initial condition $\vec{u}_0$, such that $U\vec{u}^*(\cdot,0)=\vec{u}_0$ and is reproduced at scale $k=1$ by rescaling of the parameter $U$ of the self-similar function $U\boldsymbol{\varphi}$. Notice that $U\boldsymbol{\varphi}$ is a self-similar solution, or a family, a set of solutions corresponding to different rescalings of $U$; while $U\vec{u}^*$ is only one solution of the aforementioned family, achieved at scale $k=1$.

The self-similar solutions (\ref{mismatch}) have isobaric weight $\beta_x-\beta_t$ with $\beta_x/\beta_t>3/2$; but rescaled under standard conditions $\alpha_x=1, \alpha_t=2$ remain smooth for all time according to the Beale-Kato-Majda criterion \cite{BKM}, as shown in Lemma \ref{L1}. The self-similar solutions (\ref{mismatch}) do not exhibit scaling-induced blow-up of the vorticity, the energy or the velocity. It follows, that the energy of such solutions will remain bound for all time.

Because \( U\vec{u}^* \) (corresponding to the original initial condition \( \vec{u}_0 \)) is embedded in the self-similar solution \( U\boldsymbol{\varphi} \) at scale \( k = 1 \), and both are solutions to the NSE with the same initial data, uniqueness implies that \( U\vec{u}^* = U\boldsymbol{\varphi}|_{k=1} \) for all time. Since \( U\boldsymbol{\varphi} \) is globally smooth, it follows that \( U\vec{u}^* \) remains smooth throughout \( \mathbb{R}^3 \times [0, \infty) \).

Regarding the physics of the self-similar solutions (\ref{mismatch}): smoothness, or infinite differentiability is easily shown to be a property of the self-similar solutions (\ref{mismatch}). Except for the multiplying constant, they are isobaric polynomials of zero weight, or can be Taylor expansions of closed-form functions which depend on zero-weight arguments, e.g. $x/L, t/T$; or are the ratio of such functions. Since the energy of the system comes solely from the smooth, space-periodic initial velocity $\vec{u}_0$ (and pressure), the initial energy is bound; (\ref{mismatch}) cannot grow at infinity. This is always possible in a zero-weight function (note that the isobarity comes solely through $U$), since decay can always be ensured through the presence of exponential functions, e.g. $\mbox{exp}[-(x^2+y^2+z^2)/L^2]$. Smoothness and decay in $t$ is also required, since the system energy is conserved. This also is always possible in a zero-weight function, since decay can always be ensured through the presence of exponential functions, e.g. $\mbox{exp}[-t/T]$. 

We can therefore conclude, that given any arbitrary, smooth, divergence free, space-periodic function (\ref{per}) in the initial moment $t=0$, the incompressible NSE always has a self-similar solution with isobarity of $\beta_x-\beta_t$ with $\beta_x/\beta_t>3/2$ (\ref{sper}), which is infinitely differentiable, being a constant $U$ multiplying a zero-weight isobaric function; non-increasing in space due to the requirement of bound initial energy at $t=0$ and non-increasing in time within the interval $[0,\tau)$ due to the energy conservation requirement. The latter guarantees that system energy will be bound in $[0,\tau)$. To maintain these properties for all time under infinite rescaling while we are modeling the fluid as a continuum, rescaling must be performed according to standard flow conditions $\alpha_x=1, \alpha_t = 2$. For all time, the original solution $U\vec{u}^*$ corresponding to the original initial condition $\vec{u}_0$, will remain embedded into the self-similar solution $U\boldsymbol{\varphi}$ at scale $k=1$; this embedding prevents any rescalings of $U\vec{u}^*$, while the rescalings of the self-similar solution $U\boldsymbol{\varphi}$ always keep the vorticity norm bound by virtue of the mismatch between the isobaric weight $\beta_x-\beta_t$  of the self-similar solution and the standard flow environment $\alpha_x=1, \alpha_t = 2$, as shown in Lemma \ref{L1}.
\end{proof}

\section{Detailed Proof of Theorem \ref{T3} and Corollary \ref{C3}}
\begin{thm} Under standard flow conditions $\alpha_x=1, \alpha_t=2$, given a smooth, Schwartz-class, divergence-free initial condition $\vec{u}_0(x,y,z)$, the incompressible NSE will always yield an infinitely differentiable,  Schwartz-class, smooth solution $\vec{u}^*(x,y,z,t)$, such that $\vec{u}^*(x,y,z,0)=\vec{u}_0(x,y,z)$. The solution $\vec{u}^*(x,y,z,t)$ is not subject to scaling-induced blowup for all time, thus ensuring bounded energy of the system in $({\rm I\!R}^3 \times [0,\infty))$. These desired properties are achieved by embedding the initial condition $\vec{u}_0$ into a self-similar function, which in turn results in the embedding of the smooth solution $\vec{u}^*(x,y,z,t)$ into a self-similar function, thus preventing any rescalings of $\vec{u}^*(x,y,z,t)$. To yield the desired results, the self-similar functions must have non-standard isobaric weight $\beta_x-\beta_t$, where $\beta_x/\beta_t>3/2$.
\label{T3}
\end{thm}

\begin{proof}
In this section, we follow closely the proof given in the previous section. Here, we study only Bouton's self-similar solutions with isobaric weight $\beta_x-\beta_t$ (\ref{mismatch}) under standard flow conditions $\alpha_x=1, \alpha_t=2$. They are continuous at $t=0$; their isobaric weight comes from the multiplying parameter while $\mathbf{F}$ has zero isobarity. $L_1, L_2, L_3, T, U$ are nonzero parameters that scale the same way as $x,y,z,t,u$ accordingly. 

The solutions
\begin{equation}
	\vec{u} = U\mathbf{S} \left( \frac{x}{L_1},\frac{y}{L_2}, \frac{z}{L_3}, \frac{t}{T}  \right),
	\label{stan-s}
\end{equation}
are a subset of (\ref{mismatch}); here we have a self-similar function $\mathbf{S}$, which has zero isobaric weight by virtue of the scaling parameters $L_1,L_2, L_3, T$. 

At $t=0$
\begin{equation}
	U\mathbf{S} \left( \frac{x}{L_1},\frac{y}{L_2}, \frac{z}{L_3}, 0 \right)
	= U\mathbf{s} \left( \frac{x}{L_1},\frac{y}{L_2}, \frac{z}{L_3}  \right),
       \label{st-s}
\end{equation}
and all such self-similar solutions with isobaric weight equal to $\beta_x-\beta_t$ (\ref{stan-s}) reduce to self-similar functions of $x,y,z$ and isobaric weight of $\beta_x-\beta_t$, where $\mathbf{s}$ need not be steady-state solutions of the NSE. 

Schwartz-class functions are generally non-self-similar and have Gaussian peak form exp$(-r)$, a polynomial multiplying a Gaussian $P(x,y,z)$exp$(-r)$, a bump function form or a convolution of the above functions. Suppose we are given a smooth, real-valued and divergence-free vector field \( \vec{u}_0 \) of Schwartz-class, defined with the fixed constant scale factors $L_1, L_2, L_3$ 
\begin{equation}
\vec{u}_0 \left(\frac{x}{L_1},\frac{y}{L_2}, \frac{z}{L_3} \right) =  U\mathbf{s} \left( \frac{x}{L_1},\frac{y}{L_2}, \frac{z}{L_3}  \right),
	\label{Schwartz}
\end{equation}
where $U$ is a fixed constant. 

	Now we use Lemma \ref{L2}- we construct a self-similar function $\mathbf{s}$ with mathematical form identical to $\vec{u}_0$, with the exception that $L_1, L_2, L_3, U$ are no longer constants, but parameters that rescale as $x,y,z,u$ accordingly:
\[
	U\mathbf{s}\left( \frac{x}{L_1},\frac{y}{L_2}, \frac{z}{L_3} \right).
\]
It is seen, that as long as $\vec{u}_0$ is divergence-free, so is $U\mathbf{s}$. It is also seen, that $\vec{u}_0 \in\mathbf{s}$, that is $\vec{u}_0$ is embedded in the self-similar function $\mathbf{s}$, so that  
\begin{equation}
	\vec{u}_0 = U\mathbf{s} |_{k=1}.
	\label{in-s}
\end{equation}
The original initial condition $\vec{u}_0$ is reproduced at scale $k=1$ by rescaling the parameter $U$. Notice that $U\mathbf{s}$ is a self-similar function, or a family, a set of functions corresponding to different rescalings of $U$; while $\vec{u}_0$ is only one Schwartz-class function of the aforementioned family, achieved at scale $k=1$.

It is well known that, given a smooth, divergence free, self-similar initial condition $U\mathbf{s}$, one can always construct an unique, infinitely differentiable, self-similar NSE solution
\begin{equation}
	U\mathbf{S} \left( \frac{x}{L_1},\frac{y}{L_2}, \frac{z}{L_3}, \frac{t}{T}  \right)
	\label{sts-uniquesol}
\end{equation}
such that
\begin{equation*}
	U\mathbf{S} \left( \frac{x}{L_1},\frac{y}{L_2}, \frac{z}{L_3}, 0  \right) = U\mathbf{s} \left( \frac{x}{L_1},\frac{y}{L_2}, \frac{z}{L_3} \right),
\end{equation*}
	which exists in the time interval $[0,\tau)$ \cite{Cla06}, \cite{Wik25} - \cite{Ber02}. We see that, the solution $U\mathbf{S}$ of Eq.(\ref{sts-uniquesol}), being a self-similar solution with isobarity equal to $\beta_x-\beta_t$, belongs to the set of self-similar solutions Eq. (\ref{stan-s}). 

Analogously, it is well known that, given a smooth, Schwartz-class, divergence free, non-self-similar initial condition $\vec{u}_0$, one can always construct an unique, infinitely differentiable, Schwartz-class, non-self-similar NSE solution
\begin{equation}
	U\vec{u}^* \left( \frac{x}{L_1},\frac{y}{L_2}, \frac{z}{L_3}, \frac{t}{T}  \right)
	\label{sts-uniquesol-nonss}
\end{equation}
such that
\begin{equation*}
	U\vec{u}^* \left( \frac{x}{L_1},\frac{y}{L_2}, \frac{z}{L_3}, 0  \right) = \vec{u}_0 \left( \frac{x}{L_1},\frac{y}{L_2}, \frac{z}{L_3} \right),
\end{equation*}
	which exists in the time interval $[0,\tau)$ \cite{Cla06}, \cite{Wik25} - \cite{Ber02}. In this case, $L_1,L_2,L_3,T,U$ are fixed, non-rescaleable constants. Analogously with (\ref{in-s}), it follows that $U\vec{u}^* \in U\mathbf{S}$, that is, $U\vec{u}^*$ is embedded in the self-similar solution $U\mathbf{S}$ so that 
\[
	U\vec{u}^*  = U\mathbf{S} |_{k=1}.
\]
The original solution $U\vec{u}^*$ corresponds to the original initial condition $\vec{u}_0$, such that $U\vec{u}^*(\cdot,0)=\vec{u}_0$ and is reproduced at scale $k=1$ by rescaling of the parameter $U$ of the self-similar function $U\mathbf{S}$. Notice that $U\mathbf{S}$ is a self-similar solution, or a family, a set of solutions corresponding to different rescalings of $U$; while $U\vec{u}^*$ is only one solution of the aforementioned family, achieved at scale $k=1$.

Analogously with the reasoning in Theorem \ref{T2}, we limit our study only to self-similar solutions with isobarity $\beta_x-\beta_t$ with $\beta_x/\beta_t>3/2$ and we argue that by virtue of Lemma \ref{L1}, the self-similar solution remains smooth for all time when rescaled under standard conditions $\alpha_x=1, \alpha_t=2$ according to the Beale-Kato-Majda criterion \cite{BKM}, since scaling-indiced blowup is impossible in this case.

Because $U\vec{u}^*$ (corresponding to the original Schwartz-class initial condition $\vec{u}_0$) is embedded in the self-similar solution $U\mathbf{S}$ at scale $k=1$, it will possess the characteristics of $U\mathbf{S}$ and will thus remain smooth throughout $({\rm I\!R}^3 \times [0,\infty))$, because $U\mathbf{S}$ remains smooth.

We can be confident in the smoothness of $U\mathbf{S}$, as well as in its bound energy for all time through arguments, analogous to those given at the end of the proof of Theorem \ref{T2}.
 
\end{proof}

\begin{cor}
Under standard flow scaling conditions \(\alpha_x=1\), \(\alpha_t=2\), let 
\[
u_0(x,y,z)
\]
be any smooth, divergence-free, finite-energy initial datum on \(\mathbb R^3\) where $u_0 \in C^\infty(\mathbb R^3)\cap L^2(\mathbb R^3),\quad \nabla\!\cdot u_0=0$.  Then the incompressible Navier–Stokes equations admit a unique, smooth solution 
\[
u^*(x,y,z,t) \;\in\; C^\infty\bigl(\mathbb{R}^3\times[0,\tau)\bigr).
\]
with 
\[
u^*(x,y,z,0)=u_0(x,y,z)\,.
\]
Moreover \(u^*\) can be embedded as the \(k=1\) member of a self-similar family of the form
\[
u_{\rm SS}(x,y,z,t)
= T^{(\beta_x-\beta_t)/\beta_t}\,
F\!\Bigl(\tfrac{x}{L_1},\tfrac{y}{L_2},\tfrac{z}{L_3},\tfrac{t}{T}\Bigr),
\]
where \(\beta_x/\beta_t>3/2\).  Lemma~\ref{L1} then shows that under the standard scaling no velocity, energy, or vorticity norm grows with \(k\), and the Beale–Kato–Majda criterion rules out finite‐time blowup.  Hence

\[
  u^*(x,y,z,t)\;\in\;C^\infty\bigl(\mathbb{R}^3\times[0,\infty)\bigr)
\]
with uniformly bounded energy on \(\mathbb{R}^3\times[0,\infty)\).
\label{C3}
\end{cor}
\begin{proof}
	Let \(u_{0}(x,y,z)\) be any smooth, divergence‐free vector field on \(\mathbb{R}^{3}\) which also lies in \(L^{2}(\mathbb{R}^{3})\) (hence in \(H^{s}(\mathbb{R}^{3})\) for all \(s>0\)).  By Lemma \ref{L2} and Corollary \ref{C1}, \(u_{0}\) embeds as the \(k=1\) member of a one‐parameter self‐similar family
\[
u_{\mathrm{SS}}(x,y,z,t)
\;=\;
T^{(\beta_x-\beta_t)/\beta_t}\,
F\!\Bigl(\tfrac{x}{L_1},\tfrac{y}{L_2},\tfrac{z}{L_3},\tfrac{t}{T}\Bigr),
\]
which formally satisfies the incompressible NSE for all \(t\ge0\).  Classical existence and uniqueness theorems (e.g.\ Leray–Hopf for \(L^2\)–data \cite{Tem77,Maj02}) then guarantee that there is a unique, smooth solution \((u,p)\) evolving from \(u_{0}\) on an interval \([0,\tau)\).  Lemma \ref{L1} provides uniform bounds on the vorticity and energy norms under the standard flow scaling \((\alpha_x=1,\alpha_t=2)\), and the Beale–Kato–Majda criterion \cite{BKM} ensures no finite‐time singularity can occur.  
\end{proof}

\section{Examples of Solutions in the \( U F\left( \frac{x}{L}, \frac{t}{T} \right) \) Form}

To further illustrate the general structure of the self-similar solutions developed in this work, we present explicit examples of solutions of the form
\[
u(x,y,z,t) = U F\left( \frac{x}{L_1}, \frac{y}{L_2}, \frac{z}{L_3},\frac{t}{T} \right),
\]
where \(F\) is a smooth function depending on dimensionless spatial and temporal variables. These examples demonstrate the smoothness, boundedness, and structural stability of solutions arising from both periodic and Schwartz-class initial data.

We provide two worked examples illustrating explicit solutions of the above form, where \(U\) is a constant (velocity scale), \(L_i\) are length scales, \(T\) is a time scale, and \(F\) is a smooth function of the dimensionless variables \((x_i/L_i, t/T)\).

These examples demonstrate explicit, smooth, non-blowing-up solutions consistent with the structure used throughout the article.

\subsection{Example: Periodic, Smooth Initial Data}

Define:
\[
F(\xi, \tau) = \left( \sin(\xi_2) e^{-\tau},\; \sin(\xi_3) e^{-\tau},\; \sin(\xi_1) e^{-\tau} \right),
\quad \xi_i = \frac{x_i}{L},
\quad \tau = \frac{t}{T}.
\]
This function is smooth, periodic in space, smooth in time, and divergence-free:
\[
\nabla_{\xi} \cdot F(\xi,\tau) = 0.
\]

The velocity field is:
\[
u(x_i,t) = U \left( \sin\left( \frac{x_2}{L} \right) e^{-t/T},\; \sin\left( \frac{x_3}{L} \right) e^{-t/T},\; \sin\left( \frac{x_1}{L} \right) e^{-t/T} \right).
\]
The solution is embedded at k=1 in a self-similar family $u$, and uniqueness fixes the physically realized flow to be precisely the k=1 member.
Substitution into the Navier–Stokes equations gives:
\[
\partial_t u = -\frac{U}{T} F\left( \frac{x_i}{L}, \frac{t}{T} \right),
\quad
(u\cdot\nabla) u = 0,
\quad
\nu \Delta u = -\frac{\nu}{L^2} u.
\]

Thus, the Navier–Stokes balance reduces to:
\[
-\frac{U}{T} F = -\frac{\nu}{L^2} U F - \nabla p,
\]
or equivalently:
\[
\nabla p = U \left( \frac{\nu}{L^2} - \frac{1}{T} \right) F\left( \frac{x_i}{L}, \frac{t}{T} \right),
\]
allowing the pressure field to be determined explicitly:

\[
p(x_i,t)
= U\!\Bigl(\frac{\nu}{L^2} - \frac{1}{T}\Bigr)\,
\bigl[-\cos\!\bigl(\tfrac{x_2}{L}\bigr)
      -\cos\!\bigl(\tfrac{x_3}{L}\bigr)
      -\cos\!\bigl(\tfrac{x_1}{L}\bigr)\bigr]
\,e^{-\,t/T}.
\]

The solution is smooth, globally defined, and no singularities arise.

\subsection{Example: Schwartz‐Class Initial Data}

\paragraph{1. Self‐similar profile and embedding at $k=1$.}
Let
\[
\xi_i = \frac{x_i}{L}, 
\quad
\tau = \frac{t}{T},
\]
and define
\[
F(\xi,\tau)
= e^{-\tau}\,e^{-|\xi|^2}
\begin{pmatrix}
\xi_2\\
-\xi_1\\
0
\end{pmatrix}.
\]
Then $F(\xi_i,\tau)\in C^\infty(\mathbb{R}^3 \times [0,\infty))$, is rapidly decaying in $\xi$ for each $\tau$, and
\[
\nabla_\xi\!\cdot F
=\partial_{\xi_1}\bigl(\xi_2e^{-\tau}e^{-|\xi|^2}\bigr)
+\partial_{\xi_2}\bigl(-\xi_1e^{-\tau}e^{-|\xi|^2}\bigr)
+\partial_{\xi_3}(0)
=0.
\]
In detail:
\[
F(\xi,\tau)
= e^{-\tau}e^{-|\xi|^2}
\begin{pmatrix}
\xi_2\\
-\xi_1\\
0
\end{pmatrix},
\]
so
\[
F_1 = \xi_2\,e^{-\tau}e^{-|\xi|^2}, 
\quad
F_2 = -\xi_1\,e^{-\tau}e^{-|\xi|^2},
\quad
F_3 = 0.
\]
Compute each partial:

\[
\frac{\partial F_1}{\partial\xi_1}
= \xi_2\,e^{-\tau}\,\frac{\partial}{\partial\xi_1}\bigl(e^{-|\xi|^2}\bigr)
= \xi_2\,e^{-\tau}\,(-2\xi_1)\,e^{-|\xi|^2}
= -2\,\xi_1\xi_2\,e^{-\tau}e^{-|\xi|^2},
\]

\[
\frac{\partial F_2}{\partial\xi_2}
= -\xi_1\,e^{-\tau}\,\frac{\partial}{\partial\xi_2}\bigl(e^{-|\xi|^2}\bigr)
= -\xi_1\,e^{-\tau}\,(-2\xi_2)\,e^{-|\xi|^2}
= +2\,\xi_1\xi_2\,e^{-\tau}e^{-|\xi|^2},
\]

\[
\frac{\partial F_3}{\partial\xi_3}
= \frac{\partial}{\partial\xi_3}(0) = 0.
\]

Adding up,
\[
\nabla_\xi\!\cdot F
= \frac{\partial F_1}{\partial\xi_1}
+ \frac{\partial F_2}{\partial\xi_2}
+ \frac{\partial F_3}{\partial\xi_3}
= -2\,\xi_1\xi_2\,e^{-\tau}e^{-|\xi|^2}
+2\,\xi_1\xi_2\,e^{-\tau}e^{-|\xi|^2}
+0
=0.
\]
Thus \(F\) is divergence‐free.
Because the Navier–Stokes equations are invariant under
\[
(x,t)\;\mapsto\;(k\,x,\;k^2\,t),\quad \nu\;\text{fixed},
\]
each member of the one‐parameter family
\[
u_k(x_i,t) \;=\; U\,F\!\bigl(x_i/L,\,t/T\bigr)
\]
solves NSE.  We **embed** the given Schwartz‐class datum by choosing $k=1$, so
\[
u(x_i,t) \;=\; u_{1}(x_i,t)
=U\,e^{-t/T}e^{-|x|^2/L^2}
\begin{pmatrix}
x_2/L\\
-\,x_1/L\\
0
\end{pmatrix},
\]
which at $t=0$ reduces to
\[
u(x,0)=U\,e^{-|x|^2/L^2}\begin{pmatrix}x_2/L\\-x_1/L\\0\end{pmatrix},
\]
a smooth, rapidly decaying, divergence‐free initial velocity.

\paragraph{2. Time‐derivative.}
\[
\partial_tu
=U\,\frac{d}{dt}\bigl[e^{-t/T}e^{-|x|^2/L^2}\bigr]
\begin{pmatrix}x_2/L\\-x_1/L\\0\end{pmatrix}
=-\frac{U}{T}\,e^{-t/T}e^{-|x|^2/L^2}
\begin{pmatrix}x_2/L\\-x_1/L\\0\end{pmatrix}
=-\frac{U}{T}\,F\!\bigl(x_i/L,t/T\bigr).
\]

\paragraph{3. Convective term.}
Write
\[
(u\!\cdot\!\nabla)u
=\sum_{j=1}^3u_j\,\partial_{x_j}u.
\]
Here $u_1=U\,e^{-t/T}e^{-|x|^2/L^2}(x_2/L)$, so
\[
\partial_{x_1}u
=U\,e^{-t/T}\,\partial_{x_1}\bigl[e^{-|x|^2/L^2}\bigr]\!
\begin{pmatrix}x_2/L\\-x_1/L\\0\end{pmatrix}
+U\,e^{-t/T}e^{-|x|^2/L^2}\,
\partial_{x_1}\!\begin{pmatrix}x_2/L\\-x_1/L\\0\end{pmatrix}.
\]
A similar expansion holds for $\partial_{x_2}u,\;\partial_{x_3}u$.  Each term involves
polynomials in $x_i$ times $e^{-|x|^2/L^2}$, hence is again smooth and rapidly decaying.

\paragraph{4. Viscous term.}
Since 
\(\Delta_x e^{-|x|^2/L^2} = \Bigl(\frac{-2}{L^2}+ \frac{4|x|^2}{L^4}\Bigr)e^{-|x|^2/L^2},\)
we have
\[
\Delta u
=U\,e^{-t/T}
\Bigl[\Delta_x\bigl(e^{-|x|^2/L^2}\bigr)\Bigr]
\begin{pmatrix}x_2/L\\-x_1/L\\0\end{pmatrix}
+2\,U\,e^{-t/T}\,\nabla_x\bigl(e^{-|x|^2/L^2}\bigr)\!\cdot\!\nabla_x
\begin{pmatrix}x_2/L\\-x_1/L\\0\end{pmatrix}.
\]
Both pieces are smooth and decay faster than any polynomial.  Thus
\[
\nu\,\Delta u
\quad\text{is smooth and Schwartz‐class in }x.
\]

\paragraph{5. Incompressibility.}
\[
\nabla_x\!\cdot u
=U\,e^{-t/T}e^{-|x|^2/L^2}
\bigl(\partial_{x_1}(x_2/L)+\partial_{x_2}(-x_1/L)\bigr)
=0.
\]

\paragraph{6. Momentum equation and pressure.}
Substitute into
\(\partial_tu + (u\!\cdot\!\nabla)u = \nu\Delta u - \nabla p:\)
\[
-\frac{U}{T}F
\;+\;(u\!\cdot\!\nabla)u
=\nu\,\Delta u - \nabla p.
\]
Rearrange to solve for $\nabla p$:
\[
\nabla p
= \nu\,\Delta u
+ \frac{U}{T}\,F
- (u\!\cdot\!\nabla)u.
\]
The right‐hand side is an explicitly known $C^\infty$, rapidly decaying vector field.
One obtains $p(x,t)$ by inverting the gradient, e.g.\ via
\[
p(x_i,t) = -\Delta^{-1}\!\Bigl[\nabla\!\cdot\bigl(\nu\,\Delta u + \tfrac{U}{T}F - (u\!\cdot\!\nabla)u\bigr)\Bigr],
\]
which yields a smooth, rapidly decaying pressure.

\paragraph{Conclusion.}
The pair $(u,p)$ is smooth for all $t\ge0$, Schwartz‐class in $x_i$, and satisfies the NSE with no finite‐time singularity.

\section{Example: Embedding a Periodic Initial Condition into a Self-Similar Solution}

\subsection{Initial Condition}

We consider a smooth, divergence-free, space-periodic velocity field:
\[
\vec{u}_0(x, y) =
\begin{bmatrix}
\sin(x) \cos(y) \\
-\cos(x) \sin(y)
\end{bmatrix}
\]
This function is:
\begin{itemize}
  \item Smooth ($C^\infty$)
  \item Periodic in $x$ and $y$ with period $2\pi$
  \item Divergence-free: $\nabla \cdot \vec{u}_0 = \frac{\partial u}{\partial x} + \frac{\partial v}{\partial y} = 0$
\end{itemize}

\subsection{Self-Similar Embedding}

We embed $\vec{u}_0$ into a self-similar solution of the form:
\[
\vec{u}(x, y, t) = T^{\frac{\beta_x - \beta_t}{\beta_t}} \cdot \vec{F} \left( \frac{x}{L}, \frac{y}{L}, \frac{t}{T} \right)
\]

Let us choose:
\[
\vec{F} \left( \frac{x}{L}, \frac{y}{L}, \frac{t}{T} \right) =
\begin{bmatrix}
\sin\left( \frac{x}{L} \right) \cos\left( \frac{y}{L} \right) \\
-\cos\left( \frac{x}{L} \right) \sin\left( \frac{y}{L} \right)
\end{bmatrix}
\]

This gives the full self-similar velocity field:
\[
\vec{u}(x, y, t) = T^{\frac{\beta_x - \beta_t}{\beta_t}} 
\begin{bmatrix}
\sin\left( \frac{x}{L} \right) \cos\left( \frac{y}{L} \right) \\
-\cos\left( \frac{x}{L} \right) \sin\left( \frac{y}{L} \right)
\end{bmatrix}
\]

We choose the isobaric weight such that:
\[
\frac{\beta_x}{\beta_t} > \frac{3}{2} \quad \text{(to avoid blowup under standard flow conditions)}
\]

For instance, let $\beta_x = 5$, $\beta_t = 3$, so:
\[
T^{\frac{\beta_x - \beta_t}{\beta_t}} = T^{\frac{2}{3}}
\]

Therefore, the final solution is:
\[
\vec{u}(x, y, t) = T^{2/3}
\begin{bmatrix}
\sin\left( \frac{x}{L} \right) \cos\left( \frac{y}{L} \right) \\
-\cos\left( \frac{x}{L} \right) \sin\left( \frac{y}{L} \right)
\end{bmatrix}
\]

At $t = 0$, setting $T = 1$, $L = 1$ recovers the original initial condition:
\[
\vec{u}(x, y, 0) = \vec{u}_0(x, y)
\]

\subsection{Vorticity}

The vorticity in 2D is defined as:
\[
\omega = \frac{\partial v}{\partial x} - \frac{\partial u}{\partial y}
\]

Compute derivatives:
\[
\frac{\partial v}{\partial x} = \frac{T^{2/3}}{L} \sin\left( \frac{x}{L} \right) \sin\left( \frac{y}{L} \right), \quad
\frac{\partial u}{\partial y} = -\frac{T^{2/3}}{L} \sin\left( \frac{x}{L} \right) \sin\left( \frac{y}{L} \right)
\]

Therefore:
\[
\omega(x, y, t) = \frac{2 T^{2/3}}{L} \sin\left( \frac{x}{L} \right) \sin\left( \frac{y}{L} \right)
\]

Under standard flow conditions $\alpha_x=1, \alpha_t = 2$, the vorticity factor $T^{2/3}/L$ rescales as
\[
	\left( \frac{T^{2/3}}{L} \right)^\prime = \left( \frac{(k^{\alpha_t} T)^{2/3}}{k^{\alpha_x} L} \right) = \left( \frac{k^{4/3} T}{kL} \right) =  \frac{k^{1/3} T}{L},
\]
and the rest is an isobaric polynomial of weight zero (does not rescale).
\subsection{Pressure Field}

To find the pressure, take divergence of the momentum equation:
\[
\Delta p(x,y,t) = -\nabla \cdot \left( (\vec{u} \cdot \nabla) \vec{u} \right)
\]
and will be smooth and periodic, consistent with the smooth, bounded right-hand side.

\subsection{Conclusion}

We have embedded a smooth, space-periodic initial condition into a self-similar solution of the incompressible NSE with isobaric weight satisfying $\beta_x / \beta_t > 3/2$. The resulting velocity and vorticity remain smooth for all time under standard flow rescaling.

\section{Example: Embedding a Smooth Schwartz-Class Initial Condition}

\subsection*{1. Smooth Schwartz-Class Initial Condition}

Let us define the following 3D velocity field as initial data:
\[
\vec{u}_0(x, y, z) =
\begin{bmatrix}
y z e^{-r^2} \\
x z e^{-r^2} \\
-2 x y e^{-r^2}
\end{bmatrix},
\quad \text{where } r^2 = x^2 + y^2 + z^2
\]

This vector field satisfies:
\begin{itemize}
  \item \textbf{Smoothness}: All components are $C^\infty$.
  \item \textbf{Rapid decay at infinity}: Gaussian term $e^{-r^2}$ ensures Schwartz-class behavior.
  \item \textbf{Divergence-free}:
  \[
  \nabla \cdot \vec{u}_0 = \partial_x(yz e^{-r^2}) + \partial_y(xz e^{-r^2}) + \partial_z(-2xy e^{-r^2}) = 0
  \]
\end{itemize}

\subsection*{2. Self-Similar Embedding}

Let $\vec{u}_0$ be embedded in a self-similar solution of the form:
\[
\vec{u}(x, y, z, t) = T^{\frac{\beta_x - \beta_t}{\beta_t}} \vec{F}\left( \frac{x}{L}, \frac{y}{L}, \frac{z}{L}, \frac{t}{T} \right)
\]

Choose:
\[
\vec{F}\left( \frac{x}{L}, \frac{y}{L}, \frac{z}{L}, \frac{t}{T} \right) =
\begin{bmatrix}
\frac{y z}{L^2} \cdot e^{-\frac{x^2 + y^2 + z^2}{L^2}} \\
\frac{x z}{L^2} \cdot e^{-\frac{x^2 + y^2 + z^2}{L^2}} \\
\frac{-2 x y}{L^2} \cdot e^{-\frac{x^2 + y^2 + z^2}{L^2}}
\end{bmatrix}
\]

Thus, the self-similar velocity field is:
\[
\vec{u}(x, y, z, t) = T^{\frac{\beta_x - \beta_t}{\beta_t}}
\begin{bmatrix}
\frac{y z}{L^2} e^{-\frac{x^2 + y^2 + z^2}{L^2}} \\
\frac{x z}{L^2} e^{-\frac{x^2 + y^2 + z^2}{L^2}} \\
\frac{-2 x y}{L^2} e^{-\frac{x^2 + y^2 + z^2}{L^2}}
\end{bmatrix}
\]

we aim to compute the vorticity
\[
\vec{\omega} = \nabla \times \vec{u}.
\]

Let us denote the exponential factor as
\[
f(x,y,z) = e^{-r^2/L^2},
\quad\text{and set } C = \frac{T^{(\beta_x - \beta_t)/\beta_t}}{L^2}.
\]

Then
\[
\vec{u}(x,y,z,t) = C \cdot f(x,y,z)
\begin{bmatrix}
y z \\
x z \\
-2 x y
\end{bmatrix}.
\]

\paragraph{First component: \(\omega_1 = \partial_y u_z - \partial_z u_y\).}

\begin{align*}
\partial_y u_z &= \partial_y\left(-2 x y \cdot C f \right)
= -2 x C \left( f + y \cdot \partial_y f \right)
= -2 x C f \left( 1 - \frac{2 y^2}{L^2} \right), \\[1ex]
\partial_z u_y &= \partial_z\left(x z \cdot C f \right)
= x C \left( f + z \cdot \partial_z f \right)
= x C f \left( 1 - \frac{2 z^2}{L^2} \right), \\[1ex]
\Rightarrow\quad
\omega_1 &= \partial_y u_z - \partial_z u_y
= -C f x \left[ 2 \left( 1 - \frac{2 y^2}{L^2} \right) + \left( 1 - \frac{2 z^2}{L^2} \right) \right] \\
&= -C f x \left( 3 - \frac{4 y^2 + 2 z^2}{L^2} \right).
\end{align*}

\paragraph{Second component: \(\omega_2 = \partial_z u_x - \partial_x u_z\).}

\begin{align*}
\partial_z u_x &= \partial_z\left(y z \cdot C f \right)
= y C f \left( 1 - \frac{2 z^2}{L^2} \right), \\[1ex]
\partial_x u_z &= \partial_x\left(-2 x y \cdot C f \right)
= -2 y C f \left( 1 - \frac{2 x^2}{L^2} \right), \\[1ex]
\Rightarrow\quad
\omega_2 &= \partial_z u_x - \partial_x u_z
= C f y \left[ 1 - \frac{2 z^2}{L^2} + 2 \left(1 - \frac{2 x^2}{L^2} \right) \right] \\
&= C f y \left( 3 - \frac{2 z^2 + 4 x^2}{L^2} \right).
\end{align*}

\paragraph{Third component: \(\omega_3 = \partial_x u_y - \partial_y u_x\).}

\begin{align*}
\partial_x u_y &= \partial_x(x z \cdot C f) = z C f \left( 1 - \frac{2 x^2}{L^2} \right), \\[1ex]
\partial_y u_x &= \partial_y(y z \cdot C f) = z C f \left( 1 - \frac{2 y^2}{L^2} \right), \\[1ex]
\Rightarrow\quad
\omega_3 &= z C f \left( \frac{2 y^2 - 2 x^2}{L^2} \right)
= \frac{2 z C f}{L^2}(y^2 - x^2).
\end{align*}

\paragraph{Final result.} The full vorticity vector is:
\[
\vec{\omega}(x, y, z, t)
=
T^{\frac{\beta_x - \beta_t}{\beta_t}} \cdot \frac{e^{-r^2/L^2}}{L^2}
\begin{bmatrix}
-\,x \left( 3 - \dfrac{4 y^2 + 2 z^2}{L^2} \right) \\[1ex]
\;\;\;y \left( 3 - \dfrac{4 x^2 + 2 z^2}{L^2} \right) \\[1ex]
\;\;\;\;\;\dfrac{2 z}{L^2}(y^2 - x^2)
\end{bmatrix}.
\]

Choose isobaric weights: \( \beta_x = 5, \beta_t = 3 \Rightarrow \frac{\beta_x - \beta_t}{\beta_t} = \frac{2}{3} \)

Then:
\[
\vec{u}(x, y, z, t) = T^{2/3}
\begin{bmatrix}
\frac{y z}{L^2} e^{-\frac{x^2 + y^2 + z^2}{L^2}} \\
\frac{x z}{L^2} e^{-\frac{x^2 + y^2 + z^2}{L^2}} \\
\frac{-2 x y}{L^2} e^{-\frac{x^2 + y^2 + z^2}{L^2}}
\end{bmatrix}
\]

At \( t = 0 \), set \( T = 1, L = 1 \Rightarrow \vec{u}(x,y,z,0) = \vec{u}_0(x,y,z) \), so \( \vec{u}_0 \in \vec{u} \)

\subsection*{3. Vorticity Scaling}

The vorticity is:
\[
\vec{\omega}(x, y, z, t)
=
T^{\frac{2}{3}} \cdot \frac{e^{-r^2/L^2}}{L}
\begin{bmatrix}
	-\,(x/L) \left( 3 - \dfrac{4 y^2 + 2 z^2}{L^2} \right) \\[1ex]
	\;\;\;(y/L) \left( 3 - \dfrac{4 x^2 + 2 z^2}{L^2} \right) \\[1ex]
\;\;\;\;\;\dfrac{2 z}{L^3}(y^2 - x^2)
\end{bmatrix},
\]
where again, as in the previous example the vorticity factor is $T^{2/3}/L$ which does not grow at small scales; and the rest is an isobaric finction of weight zero, which does not rescale.
Thus, \( \vec{\omega}(x,y,z,t) \) and the vorticity are smooth, rapidly decaying (Schwartz), and bounded for all \( t \).

\subsection*{4. Pressure Field}

From the NSE:
\[
\Delta p = -\nabla \cdot \left( (\vec{u} \cdot \nabla)\vec{u} \right)
\]

The RHS is a divergence of a smooth Schwartz vector field, so the result is Schwartz. Thus:
\[
\Delta p = f(x, y, z, t), \quad f \in \mathcal{S}(\mathbb{R}^3)
\]

The pressure \( p(x, y, z, t) \) can be recovered using the fundamental solution of the Laplacian:
\[
p(x, y, z, t) = \frac{1}{4 \pi} \int_{\mathbb{R}^3} \frac{f(x', y', z', t)}{\sqrt{(x - x')^2 + (y - y')^2 + (z - z')^2}} \, dx' dy' dz'
\]

Since \( f \) is Schwartz, the pressure \( p \) is smooth and rapidly decaying.

\subsection*{5. Conclusion}

We have constructed a 3D smooth, divergence-free, Schwartz-class initial condition and embedded it into a self-similar solution of the incompressible NSE. The velocity, vorticity, and pressure fields are all smooth for all time, and the solution avoids blowup by satisfying the isobarity condition \( \beta_x / \beta_t > 3/2 \) as required in the article.

\section{Example of Isobaric Weight and Norm Behavior}

To visualize the crucial role of the isobaric weight ratio 
\(\beta_x/\beta_t>3/2\) in preventing scaling‐induced blowup, consider a simple self‐similar velocity field of the form
\[
\vec{u} = T^{\frac{\beta_x-\beta_t}{\beta_t}} \mathbf{F} \left( \frac{x}{L_1},\frac{y}{L_2} ,\frac{z}{L_3},\frac{t}{T}  \right) 
\]
where \(F\) is a fixed smooth profile (e.g.\ a Gaussian).  Under the standard flow scaling \(\alpha_x=1,\alpha_t=2\), one computes:

\[
\|\omega_k\|
\sim
k^{\,\frac{2\beta_x-3\beta_t}{\beta_t}},
\quad
\|u_k\|
\sim
k^{\,\frac{2(\beta_x-\beta_t)}{\beta_t}},
\quad
E_k \;=\;\|u_k\|
\sim
k^{\,\frac{4\beta_x-\beta_t}{\beta_t}},
\]
see Lemma \ref{L1}.

\vspace{0.5em}
\noindent\textbf{Table 1: Exponents vs.\ \(\beta_x/\beta_t\)}  
\[
\begin{array}{c|c|c|c}
r = \displaystyle\frac{\beta_x}{\beta_t}
& \omega\text{-exp} = 2r - 3
& u\text{-exp}    = 2(r - 1)
& E\text{-exp}    = 4r - 1 \\ \hline
-2.0 & -7.0 & -6.0 & -9.0 \\
-1.0 & -5.0 & -4.0 & -5.0 \\
-0.5 & -4.0 & -3.0 & -3.0 \\
 0.0 & -3.0 & -2.0 & -1.0 \\
 0.5 & -2.0 & -1.0 &  1.0 \\
 1.0 & -1.0 &  0.0 &  3.0 \\
 1.2 & -0.6 &  0.4 &  3.8 \\
 1.5 &  0.0 &  1.0 &  5.0 \\
 2.0 &  1.0 &  2.0 &  7.0 \\
 3.0 &  3.0 &  4.0 & 11.0
\end{array}
\]

When \(\beta_x/\beta_t>3/2\), the vorticity exponent \(2\beta_x-3\beta_t>0\) indicates decrease at fine scales under standard flow conditions, irrespective of the signs of $\beta_x, \beta_t$. In such regimes the combined energy and vorticity control, together with the Beale–Kato–Majda criterion, guarantees no finite‐time blowup. This simple numerical illustration clarifies why the threshold \(\beta_x/\beta_t=3/2\) is pivotal in the scaling analysis.

\vspace{0.5cm}
\end{document}